\documentclass[12pt]{amsart}
\usepackage{amssymb}
\newcommand\K{\mathbb{C}}
\newcommand\C{\mathbb{C}}
\newcommand\h{\mathfrak{h}}
\newcommand\rk{\operatorname{rk}}
\newcommand\id{\operatorname{id}}
\newcommand\GL{\operatorname{GL}}
\newcommand\Ocat{\mathcal{O}}
\newcommand\m{\mathfrak{m}}
\newcommand\Res{\operatorname{Res}}
\newcommand\res{\operatorname{res}}
\newcommand\Ind{\operatorname{Ind}}
\newcommand\ind{\operatorname{ind}}
\newcommand\Fun{\operatorname{Fun}}

\newcommand\eu{\operatorname{eu}}

\newcommand\Ad{\operatorname{Ad}}
\newcommand\z{\mathfrak{z}}
\newcommand\Leaf{\mathcal{L}}
\newcommand\ad{\operatorname{ad}}

\newcommand\Z{\mathbb{Z}}
\newcommand\CC{\operatorname{C}}
\newcommand\A{\mathcal{A}}
\newcommand\cf{\operatorname{c}}
\newcommand{\hb}{h}
\newtheorem{Thm}{Theorem}[section]
\newtheorem{Prop}[Thm]{Proposition}

\newtheorem{Lem}[Thm]{Lemma}
\theoremstyle{definition}

\author{Ivan Losev}
\title{On isomorphisms of certain functors for  Cherednik algebras}
\thanks{Supported by the NSF grant DMS-0900907}
\thanks{Address: MIT, Dept. of Math., 77 Massachusetts Avenue, Cambridge MA02139, USA}
\thanks{E-mail: ivanlosev@math.mit.edu}
\begin{document}
\begin{abstract}
Bezrukavnikov and Etingof introduced some functors between the categories $\Ocat$
for rational Cherednik algebras. Namely, they defined two induction functors
$\Ind_b, \ind_\lambda$ and two restriction functors $\Res_b,\res_\lambda$. They
conjectured that one has functor isomorphisms $\Ind_b\cong \ind_\lambda,
\Res_b\cong \res_\lambda$. The goal of this paper is to prove this conjecture.
\end{abstract}
\maketitle

\section{Introduction}
The goal of this paper is to establish isomorphisms between certain functors arising
in the representation theory of rational Cherednik algebras. These functors are parabolic
induction and restriction functors introduced by Bezrukavnikov and Etingof in \cite{BE}.

Let us recall the definition of a rational Cherednik algebra that first appeared in \cite{EG}.
The base field is the field $\K$ of complex numbers. Let $\h$ be a finite
dimensional vector space and $W\subset \GL(\h)$ be a finite subgroup generated by the subset
$W\subset S$ of {\it complex reflections}. By definition, a complex reflection is an element $s\in \GL(\h)$
of finite order with $\rk(s-\id)=1$. For $s\in S$ pick elements $\alpha_s^\vee\in \operatorname{im}(s-\id)$
and $\alpha_s\in (\h/\ker(s-\id))^*\subset \h^*$ with $\langle\alpha_s,\alpha^\vee_s\rangle=2$.
Also pick a $W$-invariant map $c:S\rightarrow \K$.  Define the rational Cherednik algebra
$H(=H_{c}(W,\h))$ as the quotient of $T(\h\oplus\h^*)\# W$
by the relations
\begin{equation}\label{eq:relations1}
\begin{split}
& [x,x']=0,\\
& [y,y']=0,\\
& [y,x]=\langle y,x\rangle-\sum_{s\in S}c(s)\langle \alpha_s,y\rangle\langle \alpha_s^\vee,x\rangle s,\\
& x,x'\in \h^*, y,y'\in \h.
\end{split}
\end{equation}

We have the triangular decomposition $H= S(\h)\otimes \K W\otimes \K[\h]$. Using this decomposition
one can introduce the category $\Ocat:=\Ocat_c(W,\h)$ for $H$ as the full subcategory of the category of left $H$-modules
consisting of all modules $M$ satisfying the following two conditions:
\begin{enumerate}
\item $M$ is finitely generated as a $\K[\h]$-module.
\item $\h$ acts on $M$ locally nilpotently.
\end{enumerate}

Now pick a parabolic subgroup $\underline{W}\subset W$, i.e., the stabilizer of some point in $\h$.
The space $\h$ decomposes into the direct
sum $\h^{\underline{W}}\oplus \h_{\underline{W}}$, where $\h^{\underline{W}}$ stands for the space
of $\underline{W}$-invariants in $\h$, and $\h_{\underline{W}}$ is a unique $\underline{W}$-stable
complement to $\h^{\underline{W}}$. Consider the rational Cherednik algebra $\underline{H}^+:=H_c(\underline{W},
\h_{\underline{W}})$, where, abusing the notation, $c$ stands for the restriction of $c$ to
$S\cap \underline{W}$. Consider the category $\underline{\Ocat}^+:=\Ocat_c(\underline{W}, \h_{\underline{W}})$.

For $b\in \h,\lambda\in \h^*$ with $W_b=W_\lambda=\underline{W}$
Bezrukavnikov and Etingof in \cite{BE}, Subsection 3.5,
defined the restriction functors $\Res_b,\res_\lambda: \Ocat\rightarrow \underline{\Ocat}^+$
and the induction functors $\Ind_b,\ind_\lambda: \underline{\Ocat}^+\rightarrow \Ocat$.
The definitions will be recalled in Subsection \ref{SUBSECTION_func_def}.
The functors $\Res_b,\Ind_b$ do not depend on $b$ up to a (non-canonical) isomorphism, and the similar
claim holds for $\res_\lambda, \ind_\lambda$.   Conjecture 3.17 in \cite{BE} asserts that
there are non-canonical isomorphisms $\Res_b\cong \res_\lambda, \Ind_b\cong \ind_\lambda$.
In this paper we prove the conjecture. In particular, the conjecture implies that the functors
$\Res_b,\Ind_b$ are biadjoint. This result was obtained earlier by Shan, \cite{Shan}, under
some mild restrictions on the parameter $c$.

The paper is organized as follows.

In Section \ref{SECTION_prelim} we gather all necessary definitions and preliminary results.
In Subsection \ref{SUBSECTION_completions1} we recall the isomorphism of completions theorem
of Bezrukavnikov and Etingof, \cite{BE}, Theorem 3.2. The functors in interest are defined
using this theorem. Their definitions are recalled in Subsection \ref{SUBSECTION_func_def2}.
In Subsection \ref{SUBSECTION_completions2} we recall some other results on isomorphisms
of completions obtained in \cite{sraco} that are used in the proof of the main result.

In Section \ref{SECTION_iso_restrictions} we prove an isomorphism of the functors $\Res_b,\res_\lambda$.
In Subsection \ref{SUBSECTION_func_def2} we introduce some auxiliary functors $\Res_{b,\lambda},\res_{b,\lambda}$
such that $\Res_{b,\lambda}\cong \Res_b, \res_{0,\lambda}\cong \res_\lambda$. Our strategy is to
establish embeddings $\res_{b,\lambda}\hookrightarrow \Res_{b,0}, \res_{0,\lambda}\hookrightarrow \res_{b,\lambda}$. We can establish the latter directly, this is done in Subsection \ref{SUBSECTION_embedding3}.
However, we arrive at some convergence issue with the former embedding. To fix these issues we need
to work with algebras and modules not over $\K$ but over $R:=\K[t^{-1},t]]$. We treat this case
in Subsection \ref{SUBSECTION_Functors_R} and then establish an embedding $\res_{b,\lambda}\hookrightarrow
\Res_{b,0}$ in Subsection \ref{SUBSECTION_embedding2}.
Finally, in Subsection \ref{SUBSECTION_proof_completion} we show that the resulted embedding $\res_{0,\lambda}\hookrightarrow \Res_{b,0}$ is actually an isomorphism.

The proof of an isomorphism $\Ind_b\cong \ind_\lambda$ is similar to that of $\Res_b\cong \res_\lambda$.
In Section \ref{SECTION_iso_inductions} we explain necessary modifications.

{\bf Acknowledgements.} I am grateful to P. Etingof and I. Gordon for stimulating discussions.

\section{Preliminaries}\label{SECTION_prelim}
\subsection{Isomorphisms of completions, I}\label{SUBSECTION_completions1}
In this subsection we will recall some results from \cite{BE} related to
isomorphisms of completions of different rational Cherednik algebras.
Namely, we  define the completions $H^{\wedge_b},H^{\wedge_\lambda},
\underline{H}^{\wedge_b},\underline{H}^{\wedge_\lambda}$, where $\underline{H}:=H_c(\underline{W},\h)$
and describe isomorphisms between $H^{\wedge_b}$ (resp., $H^{\wedge_\lambda})$ and
some matrix algebra with coefficients in $\underline{H}^{\wedge_b}$ (resp, in $\underline{H}^{\wedge_\lambda}$).

Pick a point $b\in \h$ with $W_b=\underline{W}$. Let $\K[\h]^{\wedge_{Wb}}$ denote the completion
of $\K[\h]$ at $Wb$. Define the  completion $H^{\wedge_b}:=\K[\h]^{\wedge_{Wb}}\otimes_{\K[\h]}H$ of $H$ at $b$.
The space $H^{\wedge_b}$ comes equipped with a topology, and has a unique topological algebra structure
extended from $H$ by continuouty.

Similarly, we can define the completion $\underline{H}^{\wedge_b}:=\K[\h]^{\wedge_b}\otimes_{\K[\h]}\underline{H}$ of $\underline{H}$ at $b$.

A relation between $H^{\wedge_b}$ and $\underline{H}^{\wedge_b}$ is as follows.
In \cite{BE}, Subsection 3.2, for finite groups $G_0\subset G$ and an algebra $A$ containing
$\K G_0$ Bezrukavnikov and Etingof considered the right $A$-module $\Fun_{G_0}(G,A)$ of $G_0$-equivariant
maps $G\rightarrow A$. Then they defined the {\it centralizer algebra} $Z(G,G_0,A)$
as the endomorphism algebra of the right $A$-module $\Fun_{G_0}(G,A)$.
Below we write $\CC(\bullet)$ for $Z(W,\underline{W},\bullet)$.

The following proposition is a slightly modified version of
\cite{BE}, Theorem 3.2.

\begin{Prop}\label{Prop:iso1}
There is a unique continuous homomorphism $\vartheta_{b}:H^{\wedge_b}\rightarrow \CC( \underline{H}^{\wedge_b})$
such that
\begin{equation}\label{eq:def_iso1}
\begin{split}
& [\vartheta_{b}(u)f](w)=f(wu),\\
& [\vartheta_{b}(x_\alpha)f](w)= \underline{x}_{w\alpha}f(w),\\
& [\vartheta_{b}(y_a)f](w)=\underline{y}_{wa} f(w)+\sum_{s\in S\setminus W_b}\frac{2c_s}{1-\lambda_s}\frac{\alpha_s(wa)}{\underline{x}_{\alpha_s}}(f(sw)-f(w)).\\
& u,w\in W, \alpha\in \h^*, a\in \h, f\in \Fun_{\underline{W}}(W,\underline{H}^{\wedge_b}).
\end{split}
\end{equation}
This homomorphism is an isomorphism of topological algebras.
\end{Prop}
Here $x_\alpha,\underline{x}_\alpha$ denote the elements of $H, \underline{H}$ corresponding to $\alpha\in \h^*$,
$y_a,\underline{y}_a$ have the similar meaning. Of course, when one views $\frac{1}{\underline{x}_{\alpha_s}}$
as an element of $\K[\h]^{\wedge_b}$, one expands this fraction near $b$.

The completion $\underline{H}^{\wedge_b}$ is naturally isomorphic to the completion $\underline{H}^{\wedge_0,x}:=\K[\h]^{\wedge_0}\otimes_{\K[\h]}\underline{H}$. An isomorphism
$\underline{H}^{\wedge_b}\xrightarrow{\sim} \underline{H}^{\wedge_0,x}$ is given by
$w\mapsto w,\underline{x}_\alpha\mapsto \underline{x}_\alpha-\langle b,\alpha\rangle, \underline{y}_a\mapsto
\underline{y}_a$.

Similarly, one can consider the completions $H^{\wedge_\lambda}, \underline{H}^{\wedge_\lambda}$
at $\lambda\in \h^*$ with $W_\lambda=\underline{W}$. Then one has an isomorphism $\widetilde{\vartheta}_\lambda: H^{\wedge_\lambda}\rightarrow
\CC( \underline{H}^{\wedge_\lambda})$. It is given by
\begin{equation}\label{eq:def_iso2}
\begin{split}
& [\widetilde{\theta}_{\lambda}(u)f](w)=f(wu),\\
& [\widetilde{\theta}_{\lambda}(x_\alpha)f](w)= \underline{x}_{w\alpha}f(w)-\sum_{s\in S\setminus W_\lambda}\frac{2c_s}{1-\lambda_s^{-1}}
\frac{\alpha_s^\vee(wa)}{\underline{y}_{\alpha_s^\vee}}(f(sw)-f(w)),\\
& [\widetilde{\theta}_{\lambda}(y_a)f](w)=\underline{y}_{wa} f(w).
\end{split}
\end{equation}

We remark that both completions we considered were "partial" we allowed power series either
only in $x$'s or only in $y$'s. If we allow both, then the product will not be well defined.

\subsection{Definition of functors}\label{SUBSECTION_func_def}
In this subsection we will introduce  exact functors
$\Res_{b},\res_{\lambda}:\Ocat\rightarrow \underline{\Ocat}^+,
\Ind_{b}, \ind_{\lambda}:\underline{\Ocat}^+\rightarrow \Ocat$
for $b\in \h,\lambda\in\h^*$ with $W_b=W_\lambda=\underline{W}$.

We need to define some auxiliary categories of $H,\underline{H},\underline{H}^+$-modules.
For $\mu\in \h^*$  consider the category $\Ocat^{\mu}$
consisting of all $H_c$-modules $M$ satisfying
\begin{enumerate}
\item $M$ is finitely generated over $S(\h^*)$.
\item $S(\h)^W$ acts on $M$
with generalized eigencharacter $\mu$.
\end{enumerate}
It is easy to see that
$\Ocat^{0}=\Ocat$. More generally, one can consider the category $\widetilde{\Ocat}^{\mu}$
of all $H$-modules satisfying (2). Any module in $\widetilde{\Ocat}^{\mu}$ is the direct limit
of modules in $\Ocat^{\mu}$. Similarly, we have the categories $\widetilde{\underline{\Ocat}},\widetilde{\underline{\Ocat}}^+$.

Now let us recall the definitions of the functors $\Res_b,\res_\lambda,\Ind_b,\ind_\lambda$
from \cite{BE}.

First, we define $\Res_b$.
Pick $M\in \Ocat$ and consider its completion $M^{\wedge_b}:=\K[\h]^{\wedge_{Wb}}\otimes_{\K[\h]}M$
at $Wb$. Then $M^{\wedge_b}$ is an $H^{\wedge_b}$-module and hence we can consider the push-forward
$\vartheta_{b*}(M^{\wedge_b})$ that is a $\CC( \underline{H}^{\wedge_b})$-module.
There is a natural equivalence $I: \underline{H}^{\wedge_b}$-$\operatorname{Mod}\xrightarrow{\sim}
\CC(\underline{H}^{\wedge_b})$-$\operatorname{Mod}$, see \cite{BE}, Subsection 3.2.
So we get a $\underline{H}^{\wedge_b}$-module $I^{-1}\circ \theta_{b*}(M^{\wedge_b})$.
For a $\underline{H}^{\wedge_b}$-module $N'$ and $\lambda\in \h^{*\underline{W}}$ let $\underline{E}_\lambda(N')$ stand for the space of vectors annihilated by $(\underline{y}_a-\langle\lambda,a\rangle)^n$ for all $a\in \h$
and $n\gg 0$. For an $\underline{H}$-module $N$ set $$\zeta_\lambda(N):=\bigcap_{a\in \h^{*\underline{W}}}\ker(\underline{y}_a-\langle \lambda,a\rangle).$$
We set $\Res_{b}(M):=\zeta_0\circ \underline{E}_0 \circ I^{-1}\circ \theta_{b*}(M^{\wedge_b})$.

The proof of the following lemma is easy (compare with Lemma \ref{Lem:fun_iso21} below).
\begin{Lem}\label{Lem:func_iso_easy1}
For any $\lambda\in \h^{*W}$ the functor $\zeta_\lambda\circ \underline{E}_\lambda$ is an equivalence
\begin{itemize}
\item from the category of $\underline{H}^{\wedge_b}$-modules that are finitely generated over
$\K[\h]^{\wedge_0}$,
\item to the category $\underline{\Ocat}^+$.
\end{itemize}
Moreover, the functors $\zeta_\lambda\circ \underline{E}_\lambda$ are naturally isomorphic for all $\lambda\in \h^{*\underline{W}}$.
\end{Lem}

Let us construct a functor $\Ind_{b}:\underline{\Ocat}^+\rightarrow \Ocat$.
We have an equivalence $\mathcal{F}:=\vartheta_{b*}^{-1}\circ I\circ \underline{E}_0^{-1}\circ \zeta_0^{-1}$ from $\underline{\Ocat}^+$ to the category $\Ocat^{\wedge_{b}}$ of $H^{\wedge_b}$-modules that are finitely generated over $\K[\h]^{\wedge_{Wb}}$. Now for a $H^{\wedge_b}$-module $M'$ let $E_\lambda(M')$ be the generalized eigenspace
of $S(\h)^W$ with eigenvalue $\lambda:S(\h)^W\rightarrow \K$. Set $\Ind_{b}(N):=E_0\circ \mathcal{F}(N)$. 
In \cite{BE}, Subsection 3.5, it was shown that $\Ind_b(\underline{\Ocat}^+)\subset \Ocat$ (a priory,
one only sees that $\Ind_b(\underline{\Ocat}^+)\subset \widetilde{\Ocat}$) and 
that  $\Ind_{b}$ is exact and right adjoint to $\Res_{b}$.

Proceed to the definition of $\res_\lambda$.
Pick $M\in \Ocat$. Again, consider the completion $M^{\wedge_b}$. For an $H$-module  $M_1$
let $E_\lambda(M_1)$ stand for the generalized eigenspace of $S(\h)^W$ corresponding to
the character $\lambda$ in $M_1$. Consider the $H$-module $E_\lambda(M^{\wedge_b})$.
The $H$-action on this module extends to $H^{\wedge_\lambda}$. So we can consider
the push-forward $\widetilde{\theta}_{\lambda*}\circ E_\lambda(M^\wedge_b)$ and also the $\underline{H}^{\wedge_\lambda}$-module
$N':=I^{-1}\circ \widetilde{\theta}_{\lambda*}\circ E_\lambda (M^\wedge_b)$. The operators $\underline{y}_a$
act locally with generalized eigenvalue $\lambda$ on $N'$, in other words, $N'\in \widetilde{\underline{\Ocat}}^\lambda$.

The proof of the following lemma is again easy.

\begin{Lem}\label{Lem:func_iso_easy2}
The functor $\zeta_\lambda$ is an isomorphism $\widetilde{\underline{\Ocat}}^\lambda\rightarrow \widetilde{\underline{\Ocat}}^+$.
\end{Lem}

So we set $\res_{\lambda}(M):=\zeta_\lambda\circ I^{-1}\circ \tilde{\vartheta}_{\lambda*}\circ E_\lambda (M^\wedge_b)$.

To define $\ind_{\lambda}$ we reverse the procedure. Pick $N\in \underline{\Ocat}^+$. Then,
according to \cite{BE}, Corollary 3.3,  $M_0:=\widetilde{\vartheta}_{\lambda *}^{-1}\circ I\circ \zeta_\lambda^{-1}(N)\in \Ocat^\lambda$.
Set $\ind_{\lambda}(N):=E_0(M_0^{\wedge_0})$.

The functors $\res_{\lambda},\ind_{\lambda}$ were constructed
in \cite{BE}. In fact, their initial definition was quite different,
but \cite{BE}, Proposition 3.13, established an equivalence with the definition
given above. From the initial definition of \cite{BE} it follows that
$\res_{\lambda},\ind_{\lambda}$ are exact, their images lie in $\underline{\Ocat}^+,
\Ocat$, respectively, and $\ind_{\lambda}$ is \underline{left} adjoint to $\res_{\lambda}$.

\subsection{Isomorphisms of completions, II}\label{SUBSECTION_completions2}
In this section we will explain some results from \cite{sraco}.
In \cite{sraco} we worked with the homogenized versions of the algebras. More precisely,
define the $\K[\hb]$-algebra $H_{\hb}$ as the quotient of $T(\h\oplus \h^*)[\hb]\# W$ by the homogeneous
versions of the relations (\ref{eq:relations1}), namely with the third relation replaced with
\begin{equation}\label{eq:rel_homog}
[y,x]=\hb(\langle y,x\rangle-\sum_{s\in S}c(s)\langle \alpha_s,y\rangle\langle \alpha_s^\vee,x\rangle s).
\end{equation}

We can sheafify $H_{\hb}$ over $\h\oplus \h^*/W$, compare with \cite{sraco}, Subsection 2.4,
using the procedure similar to the formal microlocalization. We get a pro-coherent sheaf $\mathcal{H}_{\hb}$ of
$\K[[\hb]]$-algebras on $\h\oplus \h^*/W$.

Similarly, picking a parabolic subgroup $\underline{W}\subset W$
one can define a $\K[\hb]$-algebra $\underline{H}_{\hb}$ and sheafify it over $\h\oplus\h^*/\underline{W}$
to get a sheaf $\underline{\mathcal{H}}_{\hb}$.

Let $\pi:\h\oplus\h^*\rightarrow \h\oplus\h^*/W$ denote the quotient morphism.
Consider the locally closed subvariety of
$\h \oplus\h^*$ consisting of all points $(b,\lambda)$ with
$W_{(b,\lambda)}=\underline{W}$. Let $\Leaf$ denote the image of this
subvariety in $\h\oplus\h^*/W$. Then $\Leaf$ is a symplectic leaf of
the Poisson variety $\h\oplus\h^*/W$.

As in \cite{sraco}, Subsection 2.4, we can define the completion $\mathcal{H}_{\hb}^{\wedge_{\Leaf}}$ of
the sheaf $\mathcal{H}_{\hb}$ along $\Leaf$ and its sheaf-theoretic restriction
$\mathcal{H}_{\hb}^{\wedge_\Leaf}|_{\Leaf}$ to $\Leaf$.

Similarly, we can define the  open subvariety $\underline{\Leaf}\subset \h^{\underline{W}}\oplus \h^{*\underline{W}}
\subset (\h\oplus\h^*)/\underline{W}$ and the completion $\underline{\mathcal{H}}_{\hb}^{\wedge_{\underline{\Leaf}}}|_{\underline{\Leaf}}$. We remark that $\Leaf$
is naturally identified with the quotient $\underline{\Leaf}/\Xi$, where $\Xi:=N_W(\underline{W})/\underline{W}$.

The sheaves we have introduced come equipped with certain group actions.
First of all, let us notice that the 2-dimensional torus $(\K^\times)^2$
acts on $H_{\hb}$: $(z_1,z_2).w=w, (z_1,z_2).x=z_1 x, (z_1,z_2).y=z_2 y, (z_1,z_2).\hb=z_1 z_2 \hb,
w \in W\subset H_{\hb}, x\in \h^*\subset H_{\hb}, y\in \h\subset H_{\hb}$.
This $(\K^\times)^2$-action extends to actions on $\mathcal{H}_{\hb},\mathcal{H}_{\hb}^{\wedge_{\Leaf}},
\mathcal{H}_{\hb}^{\wedge_{\Leaf}}|_{\Leaf}$ by sheaf of algebras automorphisms.

The sheaf $\underline{\mathcal{H}}_{\hb}^{\wedge_{\underline{\Leaf}}}|_{\underline{\Leaf}}$
also carries a similar $(\K^\times)^2$-action. Moreover, $\underline{H}_{\hb}$ is acted on
by $N_W(\underline{W})$ (the action is being induced from the natural $N_W(\underline{W})$-action
on $\h\oplus\h^*$). This action again extends to $\underline{\mathcal{H}}_{\hb}^{\wedge_{\underline{\Leaf}}}|_{\underline{\Leaf}}$.

Consider the sheaf $\CC( \underline{\mathcal{H}}_{\hb}^{\wedge_{\underline{\Leaf}}}|_{\underline{\Leaf}})$
on $\underline{\Leaf}$. There is a natural action of $\Xi$ on this sheaf
by algebra automorphisms, see \cite{sraco}, Subsection 2.3. Let $\rho:\underline{\Leaf}\rightarrow \Leaf$ denote the projection
(i.e., the quotient by $\Xi$). Abusing the notation we write
$\CC( \underline{\mathcal{H}}_{\hb}^{\wedge_{\underline{\Leaf}}}|_{\underline{\Leaf}})^{\Xi}$
instead of $\rho_*(\CC( \underline{\mathcal{H}}_{\hb}^{\wedge_{\underline{\Leaf}}}|_{\underline{\Leaf}}))^{\Xi}$.
This is a sheaf of algebras on $\Leaf$.

So we have constructed  two sheaves of algebras $\mathcal{H}_{\hb}^{\wedge_{\Leaf}}|_{\Leaf},\CC( \underline{\mathcal{H}}_{\hb}^{\wedge_{\underline{\Leaf}}}|_{\underline{\Leaf}})^{\Xi}$
on $\Leaf$. These sheaves are not isomorphic but they become isomorphic if we twist one of them
by a 1-cocycle of inner automorphisms. More precisely, let us fix an open covering
$\bigcup_i U_i$ of $\Leaf$ by $(\K^\times)^2$-stable open affine subvarieties.

\begin{Prop}\label{Prop:iso_compl}
There are $(\K^\times)^2$-equivariant $\K[[\hb]]$-linear isomorphisms
$$\Theta^i: \mathcal{H}_{\hb}^{\wedge_{\Leaf}}|_{\Leaf}(U_i)\rightarrow \CC( \underline{\mathcal{H}}_{\hb}^{\wedge_{\underline{\Leaf}}}|_{\underline{\Leaf}})^{\Xi}(U_i)$$
and $(\K^\times)^2$-invariant elements $$X_{ij}\in \z^{\hb}(\CC( \underline{\mathcal{H}}_{\hb}^{\wedge_{\underline{\Leaf}}}|_{\underline{\Leaf}})^{\Xi})(U_{ij}),$$
where $U_{ij}:=U_i\cap U_j$, such that
\begin{enumerate}
\item Modulo $\hb$ the isomorphism $\Theta^i$ coincides with the natural isomorphism
$\Theta_0: (S(\h\oplus\h^*)\#W)^{\wedge_{\Leaf}}|_{\Leaf}\rightarrow \left(\CC( S(\h\oplus\h^*)\#\underline{W})^{\wedge_{\underline{\Leaf}}}|_{\underline{\Leaf}}\right)^{\Xi}$, see \cite{sraco},
Subsection 2.5.
\item
$\Theta^i=\exp(\ad X_{ij})\Theta^{j}$ for all $i,j$.
\end{enumerate}
\end{Prop}
Here for a $\K[\hb]$-algebra $A$ by $\z^{\hb}(A)$ we denote the preimage of the center of $A/\hb  A$
in $A$.

This proposition is a weaker version of Theorem 2.5.3 in \cite{sraco} (in fact, in that theorem
we have only $\K^\times$-actions, but the proof extends directly to the sheaves
with $(\K^\times)^2$-actions).

We will apply Proposition \ref{Prop:iso_compl} in the following situation. Let
$U_1:=[(\h^{\underline{W}})^r\times \h^{*\underline{W}}]/\Xi,$ $
U_2:=[\h^{\underline{W}}\times (\h^{*\underline{W}})^r]/\Xi
\subset \Leaf$, where $(\h^{\underline{W}})^r, (\h^{*\underline{W}})^r$ denote the open
subsets of all points with stabilizer exactly $\underline{W}$. Consider the algebra $\mathcal{H}_\hb^{\wedge_{\Leaf}}|_{\Leaf}(U_1)$.
We can complete $H_{\hb}$ at $b$: take the ideal $\m_{b}\subset H_{\hb}$,
compare with \cite{sraco}, Subsection 1.2, and  set $H_{\hb}^{\wedge_b}:=\varprojlim_{n\rightarrow\infty}
H_{\hb}/\m_b^n$. In fact, the natural homomorphism $H_{\hb}\rightarrow H_{\hb}^{\wedge_b}$
factors through $\mathcal{H}_\hb^{\wedge_{\Leaf}}|_{\Leaf}(U_1)$. Moreover, $H_{\hb}^{\wedge_b}$
is the completion of $\mathcal{H}^{\wedge_{\Leaf}}|_{\Leaf}(U_1)$ with respect to the ideal
analogous to $\m_b\subset H_\hb$, compare with \cite{sraco}, Subsection 2.5.

A similar construction works for $\underline{H}_{\hb}$ so we get the completion $\underline{H}_{\hb}^{\wedge_b}$. We conclude that $\Theta^1$ induces an isomorphism
$\Theta_b:H_{\hb}^{\wedge_b}\xrightarrow{\sim}\CC(\underline{H}_{\hb}^{\wedge_b})$.
We remark that this isomorphism is equivariant with respect to the
second copy of $\K^\times$ in $(\K^\times)^2$ (the one acting on $y$'s).

Note however, that we can produce an isomorphism $H_{\hb}^{\wedge_b}\rightarrow \CC( \underline{H}_{\hb}^{\wedge_0})$ by taking a homogeneous version of $\vartheta_{b}$. Namely,
define $\vartheta_{b}$ on the generators of $H_{\hb}$ by
\begin{equation}\label{eq:def_iso_h}
\begin{split}
& [\vartheta_{b}(u)f](w)=f(wu),\\
& [\vartheta_{b}(x_\alpha)f](w)= \underline{x}_{w\alpha} f(w),\\
& [\vartheta_{b}(y_a)f](w)=\underline{y}_{wa}f(w)+\hb\sum_{s\in S\setminus W_b}\frac{2c_s}{1-\lambda_s}\frac{\alpha_s(wa)}{\underline{x}_{\alpha_s}}(f(sw)-f(w)).
\end{split}
\end{equation}
Then $\vartheta_{b}$ uniquely extends to a topological algebra isomorphism
$H_{\hb}^{\wedge_b}\rightarrow \CC( \underline{H}_{\hb}^{\wedge_b})$. We remark that
$\vartheta_{b}$ is also $\K^\times$-equivariant.

\begin{Lem}\label{Lem:iso_compl1}
There is an invertible element $X\in \K[\h/\underline{W}]^{\wedge_b}$
such that $\Theta_{b}=\Ad(X)\circ \vartheta_b$.
\end{Lem}
\begin{proof}
This follows from \cite{sraco}, Lemma 5.2.1.
\end{proof}

Being $\K^\times$-equivariant both $\Theta_{b}, \vartheta_{b}$ restrict to
isomorphisms between the subalgebras in $H_{\hb}^{\wedge_b}, \CC( \underline{H}^{\wedge_0}_{\hb})$
consisting of all $\K^\times$-finite vectors ("$\K^\times$-finite" means "lying in a finite dimensional
$\K^\times$-stable subspace"). Take the quotient of these subalgebras by $\hb-1$. We get the algebras
$H^{\wedge_b}, \CC(\underline{H}^{\wedge_b})$. Let $\theta_{b}$ denote the isomorphism
of these algebras induced by $\Theta_{b}$.
We still have the equality $\theta_{b}=\Ad(X)\circ\vartheta_{b}$.

Applying the same considerations to $U_2$,  we get an isomorphism $$\widetilde{\theta}_{\lambda}: H^{\wedge_\lambda}\rightarrow \CC(\underline{H}^{\wedge_{\lambda}}).$$
and an invertible element $\widetilde{X}\in \K[\h^*/\underline{W}]^{\wedge_\lambda}$ with $\widetilde{\theta}_{\lambda}=\Ad(\widetilde{X})\circ\widetilde{\vartheta}_{\lambda}$.



\section{Isomorphism of the restriction functors}\label{SECTION_iso_restrictions}
\subsection{Functors $\Res_{b,\lambda}, \res_{b,\lambda}$}\label{SUBSECTION_func_def2}
Let $b\in \h^{\underline{W}},\lambda\in \h^{*\underline{W}}$.

Suppose $W_b=\underline{W}$.
Let us define a functor $\Res_{b,\lambda}: \Ocat\rightarrow \underline{\Ocat}^+$ by
$$\Res_{b,\lambda}(M)=\zeta_\lambda \circ I^{-1}\circ \underline{E}_\lambda\circ (\theta_b)_*(M^{\wedge_b}).$$
Here the functor $\underline{E}_\lambda$ on the category of $\CC(\underline{H})$-modules is defined
as before using the natural embedding $\K[\h^*]^{\underline{W}}\hookrightarrow \underline{H}^{\underline{W}}\hookrightarrow \CC(\underline{H})$ (see \cite{sraco},
Subsection 2.3).

\begin{Lem}\label{Lem:fun_iso1}
There is an isomorphism $\Res_{b,\lambda}\cong \Res_b$ for all $\lambda$.
\end{Lem}
\begin{proof}
First of all,  let us remark that $\underline{E}_\lambda$ and $I^{-1}$ commute.
Since $\zeta_\lambda\circ \underline{E}_\lambda$ is isomorphic to $\zeta_0\circ \underline{E}_0$
(Lemma \ref{Lem:func_iso_easy1}), we see that $\Res_{b,\lambda}\cong \Res_{b,0}$.

Recall $X$ from Lemma \ref{Lem:iso_compl1}. The existence of $X$ implies that the functors
$(\theta_b)_*$ and $(\vartheta_b)_*$ between the categories of $H^{\wedge_b}$- and of
$\CC(\underline{H}^{\wedge_b})$-modules are isomorphic. Our claim follows.
\end{proof}

In fact, it will be useful for us to rewrite the definition of $\Res_{b,\lambda}$
a little bit. Namely, for a $H^{\wedge_b}$-module $M$ let $\underline{E}_{\lambda}^{\theta}$ denote
the generalized eigenspace of the algebra $\K[\h^*/\underline{W}]$ with eigenvalue
$\lambda$, where $\K[\h^*/\underline{W}]$ acts on $M$ via $\theta_b^{-1}$. So we have
$$\Res_{b,\lambda}(M)=\zeta_\lambda \circ I^{-1}\circ  (\theta_b)_*\circ \underline{E}_\lambda^\theta(M^{\wedge_b}).$$

The definition of $\res_{b,\lambda}$ is more technical. Let $W_\lambda=\underline{W}$.

Below we will need certain "Euler elements" in $H,\underline{H},\underline{H}^+$, see \cite{GGOR}, Subsection 3.1.
Pick some basis $y_1,\ldots,y_n\in \h\subset H$ and let
$x_1,\ldots,x_n\in \h^*$ be the dual basis. We set $\eu:=\sum_{i=1}^n \frac{1}{2}(x_iy_i+y_i x_i)$.
This element does not depend on the choice of $y_1,\ldots,y_n$, is $W$-invariant
and satisfies the commutation relations $[\eu,x]=x, [\eu,y]=-y, x\in \h^*, y\in \h$.


Similarly, we can introduce the Euler elements $\underline{\eu}\in \underline{H},\underline{\eu}^+\in \underline{H}^+$.

For a topological $H$-module $M$  consider the subspace $M^{\heartsuit}\subset M$, whose elements, by definition,
are sums $\sum_{a\in \K}\sum_{i\geqslant 0}m_{a,i}$,
where
\begin{itemize} \item the first sum is finite,
\item there is $N_a$ such that  $(\eu-a-i)^{N_a}m_{a,i}=0$,
\item  and the sum $\sum_{i\geqslant 0}m_{a,i}$ converges.\end{itemize}
Then $M^{\heartsuit}$ is an $H$-submodule in $M$.
For example, let $M\in \Ocat$. Consider the completion $M^{\wedge_0}$ at 0.
The element $\eu$ acts diagonalizably on any simple module in $\Ocat$.  Since
any object in $\Ocat$ has finite length it follows that $\eu$ acts locally
finitely $M$. From here it is easy to see that $M^{\wedge_0\heartsuit}= M^{\wedge_0}$.

For $M\in \Ocat$ we set
$$\res_{b,\lambda}(M) =\zeta_\lambda\circ  I^{-1}\circ(\widetilde{\theta}_{\lambda})_*\circ E_\lambda (M^{\wedge_b\heartsuit})$$
By construction, the operators $\underline{y}_a$ act  on $I^{-1}\circ(\widetilde{\theta}_{\lambda})_*\circ E_\lambda(M^{\wedge_b\heartsuit})$ with generalized eigenvalue $\lambda$, so
$\res_{b,\lambda}(M)\in \widetilde{\underline{\Ocat}}^+$.
Later we will see that $\res_{b,\lambda}(M)$ is actually in $\underline{\Ocat}^+$.

\begin{Lem}\label{Lem:iso_easy3}
We have $\res_{0,\lambda}=\res_\lambda$.
\end{Lem}
\begin{proof}
This follows from the equality $M^{\wedge_0\heartsuit}=M^{\wedge_0}$ and the existence
of an element $\widetilde{X}\in \K[\h/\underline{W}]^{\wedge_\lambda}$, compare with the proof of
Lemma \ref{Lem:fun_iso1}.
\end{proof}

We remark that $E_\lambda(M)$ coincides with the generalized eigenspace of $S(\h)^{\underline{W}}$
with eigenvalue $\lambda$, where $S(\h)^{\underline{W}}$ acts on $M$ via $\widetilde{\theta}_\lambda^{-1}$.

Below we will show that $\res_{0,\lambda}\hookrightarrow \res_{b,\lambda}$ and $\res_{b,\lambda}\hookrightarrow \Res_b$. The first embedding is established in Subsection \ref{SUBSECTION_embedding3}. The proof is not very complicated,
although it is somewhat unsatisfactory because it works only for the field $\C$ (perhaps it should be possible
to make the same ideas work over an arbitrary algebraically closed field of characteristic 0, but we do not
know how). The embedding $\res_{b,\lambda}\hookrightarrow \Res_{b,\lambda}$ is more complicated.
Let us explain where complications come from.

Basically, we need to produce an embedding
\begin{equation}\label{eq:required_embedding}
(\widetilde{\theta}_\lambda)_*\circ E_{\lambda}\circ (\bullet^{\wedge_b\heartsuit})\hookrightarrow (\theta_{b})_*\circ \underline{E}_\lambda^\theta(\bullet^{\wedge_b})
\end{equation}
of functors  $\Ocat\rightarrow \widetilde{\underline{\Ocat}}^\lambda$.
That is, for $M\in \Ocat$ we need to construct
a functorial embedding $\Upsilon_M:E_\lambda(M^{\wedge_b\heartsuit})\rightarrow M^{\wedge_b}$ such that
$\Upsilon_M(\widetilde{\theta}_\lambda^{-1}(h).m)=\theta_b^{-1}(h).\Upsilon_M(m)$
for all $h\in \CC(\underline{H})$.

Recall the notation used in Subsection \ref{SUBSECTION_completions2}, and in particular,
isomorphisms
$$\Theta^i: \mathcal{H}_{\hb}^{\wedge_{\Leaf}}|_{\Leaf}\rightarrow \CC( \underline{\mathcal{H}}_{\hb}^{\wedge_{\underline{\Leaf}}}|_{\underline{\Leaf}})^{\Xi}(U_i), i=1,2,$$
and an element
$$X^{12}\in \z^\hb(\CC( \underline{\mathcal{H}}_{\hb}^{\wedge_{\underline{\Leaf}}}|_{\underline{\Leaf}})^{\Xi})(U_{12}),$$
with $\Theta^1=\exp(\ad X_{12})\Theta^{2}$.

Our goal will be to produce $\Upsilon_M$ from $\exp((\Theta^2)^{-1}(X_{12}))$.
A rough idea here is to make $\exp((\Theta^2)^{-1}(X_{12}))$ act on $E_{\lambda}(M^{\wedge_b\heartsuit})$
by "setting $\hb=1$". However, it is unclear why the infinite sum $\exp((\Theta^2)^{-1}(X_{12}))m$
has to converge for any $m\in E_{\lambda}(M^{\wedge_b\heartsuit})$.
In fact, we can make the sum to converge but we need to change our setting for this.
Namely, we will replace $\K$ with the field $R:=\K[t^{-1},t]]$ of formal Laurent series and a point
$(b,\lambda)$ with $(b,\lambda/t)$. In the next subsection we will see that the required sum
converges and define an embedding similar to (\ref{eq:required_embedding}). Then we will introduce
functors $\res_{b,\lambda/t},\Res_{b,\lambda/t}$ and establish
an embedding $\res_{b,\lambda/t}\hookrightarrow \Res_{b,\lambda/t}\cong \Res_{b,0/t}$. Next, in Subsection \ref{SUBSECTION_embedding2} we will see that the embedding $\res_{b,\lambda/t}\hookrightarrow \Res_{b,0/t}$
gives rise to an embedding $\res_{b,\lambda}\hookrightarrow \Res_{b,0}$.

\subsection{Functors $\Res_{b,\lambda/t},\res_{b,\lambda/t}$}\label{SUBSECTION_Functors_R}
Set $R:=\K[t^{-1},t]]$.
Consider the $R$-algebra $R[\h^*/W]:=R\otimes_\K \K[\h^*/W]$. It has a maximal
ideal $\m_{\lambda/t}$ corresponding to $\lambda/t$, so we can form the completion $R[\h^*/W]^{\wedge_{\lambda/t}}$
with respect to this ideal.
Consider the algebras $H_R:=R\otimes H,\underline{H}_R$ and the sheaves
$\mathcal{H}_{R,\hb}^{\wedge_{\Leaf}}|_{\Leaf}$, etc. The isomorphisms
$\Theta^1,\Theta^2$ naturally extend to isomorphisms of the algebras of sections
of the corresponding sheaves. Now form the algebras $H_R^{\wedge_{\lambda/t}},
\underline{H}_R^{\wedge_{\lambda/t}}$ similarly to the above. The isomorphism $\Theta_2$
induces an isomorphism $\widetilde{\theta}_{\lambda/t}: H_R^{\wedge_{\lambda/t}}\rightarrow
\CC(\underline{H}_R^{\lambda/t})$. Similarly, we have the completions $H_R^{\wedge_b},
\underline{H}_R^{\wedge_b}$ and their isomorphism $\theta_b$ induced by $\Theta^1$.

The algebras considered above come with the "$t$-Euler" derivation $t\frac{d}{dt}$.
Since $\Theta^1,\Theta^2$ are defined over $\K$, we see that they intertwine $t\frac{d}{dt}$.
It follows that $\theta_b,\widetilde{\theta}_{\lambda/t}$ also intertwine these derivations.

Now let $M\in \Ocat$. Consider the $H_R$-module $M[t^{-1},t]]$ and its completion
$M[t^{-1},t]]^{\wedge_b}$ in the $\m_b$-adic topology, where we view $\m_b$ as a maximal ideal
in $R[\h/W]$. We equip $M[t^{-1},t]]^{\wedge_b}$ with a topology taking $U_{k,l}:=t^k M^{\wedge_b}[[t]]+ \m_b^l M[t^{-1},t]], k\in \Z, l\in \Z_{\geqslant 0}$ for the fundamental system of neighborhoods of 0. In other words, a sequence $m_i$ of elements in $M[t^{-1},t]]^{\wedge_b}$ converges if the images of $m_i$ in $M[t^{-1},t]]^{\wedge_b}/\m_b^n= M/\m_b^n[t^{-1},t]]$ converge in the $t$-adic topology for all $n$.
We can define the $H_R$-submodule $E_{\lambda/t}(M[t^{-1},t]]^{\wedge_b})$ similarly to the above.
Our goal now will be to produce a certain family of maps $E_{\lambda/t}(M[t^{-1},t]]^{\wedge_b})\rightarrow
M[t^{-1},t]]^{\wedge_b}$.

 Define a derivation $d$ of $H_{R,\hb}$
by $d.w=0, d.x_{\alpha}=0, d.y_\alpha= y_\alpha, d.\hb=\hb, d.t=-t$.
The algebra $H_{R,\hb}$ acts on $M[t^{-1},t]]^{\wedge_b}$ via the homomorphism
$H_{R,\hb}\twoheadrightarrow H_R$ given by $x_\alpha\mapsto x_\alpha, y_a\mapsto
y_a, \hb\mapsto 1, w\mapsto w$.
Consider the ideal $\m$ in $\z^{\hb}(H_{R,\hb})$ corresponding to
the point $(b,\lambda/t)$.
Let $(H_{R,\hb})_{d-fin},\widetilde{\A}$ denote the subalgebras of $d$-finite elements
in $H_{R,\hb},H_{R,\hb}^{\wedge_{b,\lambda/t}}$, where the latter stands for the completion
of $H_{R,\hb}$ with respect to $\m$.

\begin{Prop}\label{Prop:convergence}
For any $m\in E_{\lambda/t}(M[t^{-1},t]]^{\wedge_b})$
the  map $(H_{R,\hb})_{d-fin}\rightarrow
M[t^{-1},t]]^{\wedge_b}), h\mapsto h.m$ extends uniquely to a continuous map
$\widetilde{\A}\rightarrow M[t^{-1},t]]^{\wedge_b}$.
\end{Prop}
\begin{proof}
We need to show that for all $a\in \Z, n_1,n_2\in \Z_{\geqslant 0}$ there is $n$ such that
$(H_{R,\hb}\m^n)_a. m\subset U_{n_1,n_2}$, where $(H_{R,\hb}\m^n)_a$ denotes the subspace of all elements
$f\in H_{R,\hb}\m^n$ with $d(f)=a f$ (we remark that $\m$ is $d$-stable).
First of all, let us define some $d$-stable filtration
on $H_{R,\hb}$ that is equivalent to $H_{R,\hb}\m^n$. Choose elements $f_1,\ldots, f_k$
generating the ideal of $b$ in $\K[\h]^W$ and elements $g_1,\ldots g_r\in R[\h^*]^W$ generating the ideal
of $\lambda/t$. The latter ideal is $d$-stable, so we may assume that all $g_i$ are eigenvectors for $d$
with some integral eigenvalues $\alpha_1,\ldots,\alpha_r$. The $R$-algebra $\z^{\hb}(H_{R,\hb})$ is finite over its subalgebra
generated by $f_i,g_i, i=1,\ldots,r$ and $\hb$. Let $F_1,\ldots,F_l$ be a finite set of generators
that are eigenvectors for $d$ with eigenvalues, say, $\beta_1,\ldots,\beta_l$.
Then it is easy to see that $H_{R,\hb}\m^n$ is equivalent to the  filtration
$\m_i$  defined as follows:
$$\m_i:=\sum_{j+k+s=i}f_{i_1}\ldots f_{i_j}\operatorname{Span}_R(F_1,\ldots F_l)\hb^{s}g_{i_1'}\ldots g_{i'_k}.$$
Consider a monomial $f:=\hb^{s}\lambda^q f_{i_1}\ldots f_{i_k} F_t g_{i'_1}\ldots g_{i'_l}\in \m_i$
such that $d$ acts on $f$ by $a$. The last condition can be rewritten as $s-q+\beta_t+\sum_{j=1}^l \alpha_{i'_j}=a$.
For sufficiently large $l$, say $l>l_1$, where $l_1$ depends only on $m$, we have $g_{i'_1}\ldots g_{i'_l} m=0$.
So we may assume that $l\leqslant l_1$. Also if $k\geqslant n_2$, then $f_{i_1}\ldots f_{i_k}M[t^{-1},t]]^{\wedge_b}\subset \m_b^{n_2}M[t^{-1},t]]^{\wedge_b}\subset U_{n_1,n_2}$. So we may assume that
$k\leqslant n_2$. This means that $s\geqslant i-n_2-l_1$ and so $q\geqslant i-M$, where $M$ is some constant
depending only on $n_2,l_1$. The $\K$-linear span of all vectors of the form $\lambda^{-q}f.m$ for all monomials
$f$ with $l\leqslant l_1, k\leqslant n_2$ is finite dimensional (recall that $\hb$ acts by 1).
It follows that for sufficiently large $i$ we get $\lambda^q (\lambda^{-q}f.m)\in U_{n_1,n_2}$.
\end{proof}

Set $A_\hb:=(\Theta^2)^{-1}(X_{12})$.  Let us view $A_\hb$ as an element of $\widetilde{\A}$. It is
annihilated by $d$. So it is also annihilated by $t\frac{d}{dt}$ modulo $\m$ and hence lies in $\K$ modulo $\m$. Subtracting the corresponding element of $\K$, we may assume that $A_\hb\in \m$. So the element $\exp(A_\hb)\in \widetilde{A}$ is defined
and is $d$-invariant as well. So it defines a linear map $E_{\lambda/t}(M[t^{-1},t]]^{\wedge_b})\rightarrow
M[t^{-1},t]]^{\wedge_b}$.

Moreover, let $f\in \CC(\underline{H}_{R,\hb})$. Then $f$ is $\underline{d}$-finite,
where $\underline{d}$ is the derivation of $\CC(\underline{H}_{R,\hb})$ defined similarly
to $d$. The isomorphisms $\Theta^1,\Theta^2: H_{R,\hb}^{\wedge_{b,\lambda/t}}\rightarrow
\CC(\underline{H}_{R,\hb}^{\wedge_{b,\lambda/t}})$ both intertwine $d$ and $\underline{d}$.
It follows that $(\Theta^1)^{-1}(f)\exp(A_h)=\exp(A_h)(\Theta^2)^{-1}(f)$ in $H_{R,\hb}^{\wedge_{b,\lambda/t}}$
and so the actions of the two sides on $E_{\lambda/t}(M[t^{-1},t]]^{\wedge_b})$ agree.
But $(\Theta^{-1})(f)$ acts as $\theta_b^{-1}(f_1)$, while $(\Theta^2)^{-1}(f)$
acts as $\widetilde{\theta}_{\lambda/t}^{-1}(f_1)$. Here $f_1$ is the image of $f$ in
$\CC(\underline{H}_R)$. We remark that any $d$-finite element  of $\CC(\underline{H}_R)$
is represented in this form. Set $\Upsilon_{M,t}(m):=\exp(A_h).m$. We conclude that
\begin{equation}\label{eq:map_equality_t}
\theta_b^{-1}(h)\Upsilon_{M,t}(m)=\Upsilon_{M,t}(\widetilde{\theta}_{\lambda/t}^{-1}(h)m),
\end{equation}
for all $d$-finite (and then, automatically, for all) elements of $\CC(\underline{H}_R)$.

So we get the map $\Upsilon_{M,t}: E_{\lambda/t}(M[t^{-1},t]]^{\wedge_b})\rightarrow
M[t^{-1},t]]^{\wedge_b}$. Thanks to (\ref{eq:map_equality_t}), the image of this map
is contained in $\underline{E}_{\lambda/t}^\theta(M[t^{-1},t]]^{\wedge_b})$. We claim that
$\Upsilon_{M,t}$ is a bijection $E_{\lambda/t}(M[t^{-1},t]]^{\wedge_b})\rightarrow
\underline{E}_{\lambda/t}^\theta(M[t^{-1},t]]^{\wedge_b})$. Indeed, analogously to
Proposition \ref{Prop:convergence}, one can prove that the action of
$(\Theta^1)^{-1}(\CC(\underline{H}_{R,\hb})_{\underline{d}-fin})$ on $\underline{E}_{\lambda/t}^\theta(M[t^{-1},t]]^{\wedge_b})$
extends to that of $\widetilde{A}$. Then it is easy to see that the map
$m\mapsto \exp(-A_h).m$ is inverse to $\Upsilon_{M,t}$.

Also we remark that $M[t^{-1},t]]$ comes equipped with an endomorphism $t\frac{d}{dt}$. This endomorphism
extends to $M[t^{-1},t]]^{\wedge_b}$ by continuouty, the extension will be denoted by $\eu_t^{M}$. It is
compatible with the derivation $t\frac{d}{dt}$ of $H_{R,\hb}$ in the sense that $\eu_t^M(fm)=(t\frac{d}{dt}f)m+f\eu_t^M m$
for all $m\in M[t^{-1},t]]^{\wedge_b}, f\in H_{R,\hb}$. It is easy to see that both
$E_{\lambda/t}(M[t^{-1},t]]^{\wedge_b})$ and $\underline{E}_{\lambda/t}^\theta(M[t^{-1},t]]^{\wedge_b})$
are $\eu_t^M$-stable. Since $\frac{d}{dt}A_h=0$, we see that $\Upsilon_{M,t}$ intertwines the operators
$\eu_t^M$.

Now let us define certain functors $\Res_{b,\lambda/t},\res_{b,\lambda/t}: \Ocat\rightarrow \underline{\Ocat}^+_R$. Here $\underline{\Ocat}^+_R$ stands for the category of $\underline{H}_R^+$-modules $N$ equipped
with an operator $\eu_t^N$ subject to the following conditions:
\begin{enumerate}
\item $N$ is finitely generated  over $R[\h_{\underline{W}}]$,
\item the operators $y_a,a\in \h_{\underline{W}}$ act locally nilpotently on $N$,
\item the operator $\eu_t^N$ is compatible with the derivation $t\frac{d}{dt}$ of $\underline{H}_R^+$.
\end{enumerate}
 Then we will establish  isomorphism $\res_{b,\lambda/t}\xrightarrow{\sim} \Res_{b,\lambda/t}\xrightarrow{\sim}\Res_{b,0/t}$.

Let us construct a functor $\Res_{b,\lambda/t}$.

Take a module $M\in \Ocat$. Form the $H_{R}$-module $M[t^{-1},t]]^{\wedge_b}$
and the endomorphism $e_t^M$ of this module.
Consider the $\CC(\underline{H}_R^{\wedge_b})$-module $(\theta_b)_*(M[t^{-1},t]]^{\wedge_b})$.
Set $\underline{\eu}_t^{M,\theta}:=(\theta_b)_*(\eu_t^M)$. Since $\theta_b$ intertwines the derivations
$t\frac{d}{dt}$, we see that $\underline{\eu}_t^{M,\theta}$ is compatible with $t\frac{d}{dt}$.

Consider the subspace
\begin{equation}\label{eq:module} \zeta_{\lambda/t}\circ \underline{E}_{\lambda/t}\circ I^{-1}\circ(\theta_b)_*(M[t^{-1},t]]^{\wedge_b})\end{equation}
in $(\theta_b)_*(M[t^{-1},t]]^{\wedge_b})$. This subspace is $\underline{H}^+_R$-stable.
Also this subspace is preserved by $\underline{\eu}_t^{M+}:=\underline{\eu}_t^{M,\theta}+\lambda/t$. Indeed, (\ref{eq:module}) consists
of all elements that are annihilated by $y_a-\langle a,\lambda\rangle/t$ with $a\in \h^{\underline{W}}$
and by some powers of $y_a$ with $a\in \h_{\underline{W}}$. Both these conditions are
preserved by $\underline{\eu}_t^{M+}$.

Let us check that (\ref{eq:module}) lies in $\underline{\Ocat}_R^+$. For this we will need the following lemma
that is a ramification of Lemma \ref{Lem:func_iso_easy1}. Let $\underline{\Ocat}_R'$ denote the
category of all $\underline{H}_R':=\K[\h][t^{-1},t]]^{\wedge_b}\otimes_{\K[\h]}\underline{H}$-modules
$M'$ that are finitely generated over $\K[\h][t^{-1},t]]^{\wedge_b}$ and come equipped
with an operator $\eu_t^{M'}$ that is compatible with $t\frac{d}{dt}$.

\begin{Lem}\label{Lem:fun_iso21}
The assignment $(M', \eu_t^{M'})\mapsto (\zeta_{\lambda/t}\circ \underline{E}_{\lambda/t}(M'), \eu_t^{+}:=\eu_t^{M'}+\lambda/t)$ defines an equivalence between $\underline{\Ocat}_R'$
and $\underline{\Ocat}_R^+$. This equivalence does not depend on $\lambda$ up to an isomorphism.
\end{Lem}
\begin{proof}
An isomorphism $\zeta_{0}\circ \underline{E}_{0}(M')\rightarrow\zeta_{\lambda/t}\circ \underline{E}_{\lambda/t}(M')$
is given by $m\mapsto e^{\lambda/t}m$ (this map is well-defined by the definition of
$\underline{\Ocat}_R'$). So it remains to prove that $\zeta_0\circ \underline{E}_0(M')$ is a finitely
generated over $R[\h_{\underline{W}}]$. For this it suffices to check that $\underline{E}_0(M')$
is finitely generated over $R[\h]$.  To prove this one first shows that $\underline{\eu}$
acts locally finitely on $\underline{E}_0(M')$. Then the proof of that $\underline{E}_0(M')$
is finitely generated is easy, compare with \cite{BE}, the proof of Theorem 2.3.
\end{proof}

So (\ref{eq:module}) is indeed an object of $\underline{\Ocat}_R^+$. By $\Res_{b,\lambda/t}$
we denote a functor that assigns (\ref{eq:module}) to $M$. By Lemma \ref{Lem:fun_iso21},
this functor does not depend on $\lambda$.

Now let us proceed to defining the functor $\res_{b,\lambda/t}$. Again, we consider the $H_R$-module
$M[t^{-1},t]]$ and then its completion $M[t^{-1},t]]^{\wedge_b}$. Consider the $H_R$-submodule
$E_{\lambda/t}(M[t^{-1},t]]^{\wedge_b})$ of $M[t^{-1},t]]^{\wedge_b}$. It is
straightforward to see that this module is stable under $\eu_t^M$.
Recall that $\tilde{\theta}_{\lambda/t}$ intertwines the derivations $t\frac{d}{dt}$.
We have the operator $\underline{\tilde{\eu}}_t^{M}:=(\tilde{\theta}_{\lambda/t})_*(\eu_t^M)$ on $(\widetilde{\theta}_{\lambda/t})_*(E_{\lambda/t}(M[t^{-1},]]^{\wedge_b}))$
compatible with $t\frac{d}{dt}$.

Consider the subspace
\begin{equation}\label{eq:module2}
\zeta_{\lambda/t}\circ I^{-1}\circ(\widetilde{\theta}_{\lambda/t})_*(E_{\lambda/t}(M[t^{-1},]]^{\wedge_b}))
\end{equation}
in $(\widetilde{\theta}_{\lambda/t})_*(E_{\lambda/t}(M[t^{-1},]]^{\wedge_b}))$.
It comes equipped with the operator $\underline{\tilde{\eu}}_t^{+M}$ defined similarly
to $\underline{\eu}^{+M}_t$. Then we can define $\res_{b,\lambda/t}$ similarly
to $\Res_{b,\lambda/t}$.

Recall that we have the isomorphism $$\Upsilon_{M,t}: (\widetilde{\theta}_{\lambda/t})_*(E_{\lambda/t}(M[t^{-1},]]^{\wedge_b}))\rightarrow
\underline{E}_{\lambda/t}\circ(\theta_b)_*(M[t^{-1},t]]^{\wedge_b}).$$
By the construction this isomorphism  intertwines the operators
$\underline{\eu}_t^{+M},\underline{\tilde{\eu}}_t^{M+}$.
Therefore it induces an isomorphism $\res_{b,\lambda/t}\xrightarrow{\sim} \Res_{b,\lambda/t}$.

Our conclusion is that $\res_{b,\lambda/t}\xrightarrow{\sim}\Res_{b,0/t}$.

\subsection{An embedding $\res_{b,\lambda}\hookrightarrow \Res_{b,0}$}\label{SUBSECTION_embedding2}
First of all, let us relate $\Res_{b,0/t}$ and $\Res_{b,0}$.
We remark that $M^{\wedge_b}$ is nothing else as the 0-eigenspace for $\eu^M_t$ in $M[t^{-1},t]]^{\wedge_b}$.
From here, tracking the constructions of $\Res_{b,0/t},\Res_{b,0}$, we see that
$\Res_{b,0}(M)$ is the 0-eigenspace of $\underline{\eu}_t^{+M}$ in (\ref{eq:module}).
The latter subspace is $\underline{H}_R^+$-stable.
Moreover, it is easy to see that $\underline{E}_0\circ (\theta_b)_*(M[t^{-1},t]]^{\wedge_b})=
R\otimes \underline{E}_0\circ (\theta_b)_*(M^{\wedge_b})$.
Therefore $\Res_{b,0/t}(M)=R\otimes \Res_{b,0}(M)$.
In particular, $\Res_{b,0}(M)=\Res_{b,0/t}(M)_{fin}/(t-1)$,
where the subscript ``$fin$'' denotes the subspace of all $\eu_{t}^{+M}$-finite elements.

So we need to produce a functorial embedding $\res_{b,\lambda}(M)\hookrightarrow \res_{b,\lambda/t}(M)_{fin}/(t-1)$.
For this we will need to technical lemmas concerning Euler elements.

\begin{Lem}\label{Lem:el_t}
 $\Theta^2(\eu)-\underline{\eu}\in \K \hb$.
\end{Lem}
\begin{proof}
The center of $\underline{H}_\hb^{\wedge_0,x}$ coincides with $\K[[\hb]]$,
this follows easily from the claim (see \cite{BG}) that the center of $\underline{H}$ is $\K$.
Therefore the centers of $\underline{H}_\hb^{\wedge_b}$ and $\CC(\underline{H}_{\hb}^{\wedge_b})$ coincide
with $\K[[\hb]]$. Let us show that $[\Theta^2(\eu),\cdot]=[\underline{\eu},\cdot]$. The derivation $[\eu,\cdot]$
is the image of $1\in \K$ under the $\K^\times$-action on $H_\hb$ given by $z.x=zx, z.y=z^{-1}y, z.w=w, z.\hb=\hb$.
A similar claim holds for $\underline{\eu}$. The required equality follows from the claim that $\Theta^2$ intertwines
the corresponding $\K^\times$-actions. Now consider the $\K^\times$-actions induced by the gradings on $H,\underline{H}$.
They are also intertwined by $\Theta^2$. Since both $\eu$ and $\underline{\eu}$ have degree 2 with respect to
these actions, we see that $\Theta^2(\eu)-\underline{\eu}\in \K \hb$.
\end{proof}

Let $\alpha\in \K$ be such that $\Theta^2(\eu)=\underline{\eu}+\alpha\hb$.
It follows that $\widetilde{\theta}_{\lambda/t}(\eu)=\underline{\eu}+\alpha$.

\begin{Lem}\label{Lem:euler1}
Let $M$ be a $\underline{H}$-module. Then $\underline{\eu}$ acts as $\underline{\eu}^++\lambda+\dim \h^{\underline{W}}/2$
on $\z_{\lambda}(M)$.
\end{Lem}
\begin{proof}
Pick a basis $y_1,\ldots,y_n$ in such a way that $y_1,\ldots,y_k$ is a basis in $\h_{\underline{W}}$,
while $y_{k+1},\ldots,y_n$ is a basis in $\h^{\underline{W}}$. So we see that
$\underline{\eu}=\underline{\eu}^++\sum_{i=k+1}^n \frac{1}{2}(x_i y_i+y_i x_i)=
\underline{\eu}^++\sum_{i=k+1}^n x_i y_i+\frac{n-k}{2}$. The element
$\sum_{i=k+1}^n x_i y_i$ acts by $\sum_{i=k+1}x_i \langle\lambda,y_i\rangle=\lambda$
on $\z_{\lambda}(M)$. Hence our claim.
\end{proof}

Pick a section  $\varphi:\K/\Z\hookrightarrow \K$ of the natural projection $\K\twoheadrightarrow \K/\Z$.
Define an embedding $\iota:M^{\wedge_b\heartsuit}\hookrightarrow M[t^{-1},t]]^{\wedge_b}$ by sending
a sum $\sum_{i\in \Z_{\geqslant 0}}m_{\alpha,i}$ (in the notation of Subsection \ref{SUBSECTION_func_def2})
to $\sum_{i\in \Z_{\geqslant 0}}t^{\varphi(\alpha)-\alpha-i}m_{\alpha,i}$. It is easy to
see that this embedding induces an embedding $\iota:\res_{b,\lambda}(M)\hookrightarrow
\res_{b,\lambda/t}(M)$. Moreover, the last embedding becomes an $\underline{H}^+$-module
homomorphism if we modify the action of $\underline{H}^+$ on the image by $w.\iota(m)=w\iota(m),
x.\iota(m)=t^{-1}x\iota(m), y.\iota(m)=ty\iota(m), w\in \underline{W}, x\in \h^*_{\underline{W}},
y\in \h_{\underline{W}}$. We remark that $\underline{\eu}^++ \underline{\eu}_t^{+M}$
acts locally finitely on $\iota(\res_{b,\lambda}(M))$. Indeed, thanks to Lemmas \ref{Lem:el_t},
\ref{Lem:euler1}, $\underline{\eu}^++\underline{\eu}_t^{+M}$ coincides (up to adding a scalar)
with the operator on $\res_{b,\lambda}(M)$ induced by $\eu+\eu_t^M$. But
if $m=\sum_{i\geqslant 0}m_{\alpha,i}$, then $\iota(m)$ is a generalized eigenvector
for $\eu+\eu_t^M$ with eigenvalue $\varphi(\alpha)$. In particular, since $\underline{\eu}^+$
acts locally finitely on any object in $\underline{\Ocat}_R^+$, we see
that $\iota(\res_{b,\lambda}(M))\subset \res_{b,\lambda/t}(M)_{fin}$.
Now consider the induced map $\K[t^{-1},t]\otimes \iota(\res_{b,\lambda}(M))\rightarrow \res_{b,\lambda/t}(M)_{fin}$.
The kernel of this  map is stable with respect to
$\underline{\eu}^++\underline{\eu}^{+M}_t$ and the action of this operator
on $\K[t^{-1},t]\otimes \iota(\res_{b,\lambda}(M))$ is locally finite.
So let $v$ be an eigenvector in the kernel. But the kernel  is stable under the
multiplication by elements from $\K[t^{-1},t]$ as well. However, it is easy to
see that for an appropriate $k$ we have $t^k v\in \iota(\res_{b,\lambda}(M))$.
Contradiction.


So let us pick $M\in \Ocat(\cf)$. Let us introduce an embedding of $\res_{b,\lambda}(M)
\hookrightarrow \res_{b,\lambda/t}(M)$.

\subsection{An embedding $\res_{0,\lambda}\hookrightarrow \res_{b,\lambda}$}\label{SUBSECTION_embedding3}
Tracking the construction of $\res_{b,\lambda}$ we see that we need to prove the following claim:

\begin{itemize}
\item[(*)] There is a functorial embedding $E_{\lambda}(M^{\wedge_0})\hookrightarrow E_{\lambda}(M^{\wedge_b\heartsuit})$.
\end{itemize}

In fact we will show a weaker result.

\begin{Prop}\label{Prop_emb2}
There is a sufficiently small $W$-stable neighborhood $U$ of zero in $\h$ such that  for all $b\in U$
there is a functorial homomorphism $E_{\lambda}(M^{\wedge_0})\rightarrow E_\lambda(M^{\wedge_b\heartsuit})$
that is injective when $M$ is projective.
\end{Prop}
\begin{proof}
Let $U$ be a convex $W$-stable neighborhood of $0$ in $\h$.
Set $H(U):=\K[U]\otimes_{\K[\h]} H$, where $\K[U]$
stands for the algebra of analytic functions on $U$. For $M\in \Ocat$
set $M(U):=\K[U]\otimes_{\K[\h]}M$. We can define $E_\lambda(M(U))$ similarly to the above.
We have a natural homomorphism $E_{\lambda}(M(U))\hookrightarrow E_{\lambda}(M^{\wedge_0})$
(restricting a section to the formal neighborhood of 0). 

\begin{Lem}\label{Lem:embedding20}
The natural map $M(U)\rightarrow M^{\wedge_0}$ is an embedding for sufficiently small $U$. 
\end{Lem}
\begin{proof}
Since both $M\mapsto M^{\wedge_0}, M\mapsto M(U)$ are exact functors, 
it is enough to check the claim when $M$ is irreducible. Since there are only finitely many irreducibles,
it is enough to check the claim for any fixed $M$. Here it follows easily from the observation 
that $M(U)$ is a finitely generated and hence Noetherian $\K[U]$-module. 
\end{proof}

Also we have a homomorphism
$E_\lambda(M(U))\rightarrow E_{\lambda}(M^{\wedge_b})$ (the restriction to a formal neighborhood
of $b$) for any $b\in U$. Any section $s$ of $M(U)$ is a sum $\sum_{a\in \K}\sum_{i=0}^\infty m_{a,i}$, where
the meaning of $m_{a,i}$ is the same as in the definition of $\bullet^{\heartsuit}$ in Subsection \ref{SUBSECTION_func_def2},
that converges on $U$. It follows that the image of the restriction embedding $M(U)\hookrightarrow
M^{\wedge_b}$ lies in $M^{\wedge_b\heartsuit}$. So $E_\lambda(M(U))$ actually maps to
$E_{\lambda}(M^{\wedge_b\heartsuit})$.


Now we claim that for sufficiently small $U$ the embedding $E_\lambda(M(U))\rightarrow E_{\lambda}(M^{\wedge_0})$
is actually an isomorphism. The category $\Ocat$ is a length category, has enough projectives, and the number of
projectives is finite. The functor $M\mapsto E_\lambda(M(U))$ is obviously left exact, while the functor
$M\mapsto E_\lambda(M^\wedge_0)$ is exact. Recall that the latter follows 
from the fact that $\res_{0,\lambda}$ is exact, see \cite{BE}, Subsection 3.5. Therefore the embedding $E_\lambda(M(U))\hookrightarrow
E_\lambda(M^{\wedge_0})$ is an isomorphism if and only if it is an isomorphism for any projective $M$.

Now it is known, see \cite{GGOR}, that any projective has a filtration by standard modules
and, in particular, is a free $\K[\h]$-module. The action of $y_a$ gives rise to a flat
$W$-equivariant connection on this bundle.  The connection has regular singularities on the reflection
hyperplanes. Suppose that $v$ is an element of $M^{\wedge_0}$ such that $(y_a-\langle\lambda,a\rangle)v=u'(a)$
for some linear map $u':\h\rightarrow M(U)$. Since the connection given by $a\mapsto y_a$ has regular
singularities, we see that $v$ extends to some smaller neighborhood $U'$ of $0$. But the module
$E_\lambda(M^{\wedge_0})$ is finitely generated. So using an easy induction and shrinking
$U$ if necessary, we see that all generators of $E_\lambda(M^{\wedge_0})$ extend to $U$
proving our claim.

Since any projective module is free over $\K[\h]$, we see that the natural map $M(U)\rightarrow M^{\wedge_b}$
is injective, provided $M$ is projective. 
\end{proof}

\subsection{Completion of the proof}\label{SUBSECTION_proof_completion}
Let us complete the proof of the theorem. We have a homomorphism $$\res_\lambda\cong\res_{0,\lambda}\rightarrow
\res_{b,\lambda}\hookrightarrow \Res_{b,\lambda}\cong  \Res_b$$
for $b$ sufficiently close to 0. But all functors $\Res_b$ are isomorphic, see \cite{BE}, Subsection 3.7, so we have
a homomorphism $\res_{\lambda}\rightarrow \Res_b$ for all $b$. Moreover, $\res_\lambda(M)\hookrightarrow \Res_b(M)$
for any projective $M$.
As Bezrukavnikov and Etingof checked in \cite{BE}, Subsection 3.6, on the level of the Grothendieck
groups the functors $\res_\lambda,\Res_b$ are the same
($K_0(\Ocat)=K_0(W-\operatorname{Mod}), K_0(\underline{\Ocat}^+)=K_0(\underline{W}-\operatorname{Mod})$
and both $\Res_b,\res_\lambda$ produce the restriction map induced by the embedding $\underline{W}\hookrightarrow
W$). So $\res_\lambda(M),\Res_b(M)$ have the same class in the Grothendiecck group for any $M$.
In particular, $\res_\lambda(M)=\Res_b(M)$ for projective $M$. Now we have a natural transformation 
of two exact functors that is an isomorphism on projectives. Such a transformation is necessarily
an isomorphism. 


\section{Isomorphism of the induction functors}\label{SECTION_iso_inductions}
First of all, we define auxiliary functors
\begin{align}\label{eq:Ind}
\Ind_{b,0}(N)=E_0\circ (\theta_b)^{-1}_*\circ I\circ  (\zeta_0\circ \underline{E}_0)^{-1}(N),\\\label{eq:ind}
\ind_{b,\lambda}(N)=E_0(\left((\widetilde{\theta}_\lambda)_*^{-1}\circ I\circ \zeta_\lambda^{-1}(N)\right)^{\wedge_b})
\end{align}
from $\underline{\Ocat}^+\rightarrow \widetilde{\Ocat}$, where we consider $\zeta_0\circ \underline{E}_0$
as an equivalence  $\underline{\Ocat}^{\wedge_b}\rightarrow \underline{\Ocat}^+$.
We remark that we do not need to apply $\bullet^\heartsuit$ in the definition of $\ind_{b,\lambda}$.
Indeed, $E_0(M')\subset M'^\heartsuit$ for any topological $H$-module because $\eu$
acts locally finitely on any object of $\widetilde{\Ocat}$.

As in Subsection \ref{SUBSECTION_func_def2} one shows that $\Ind_{b,0}\cong \Ind_b$, while
$\ind_{0,\lambda}\cong \ind_{\lambda}$. Then, similarly, to Subsection \ref{SUBSECTION_proof_completion},
it is enough to show that there are 
\begin{itemize}\item[(A)] A homomorphism
 $\ind_{0,\lambda}\rightarrow \ind_{b,\lambda}$ that is an embedding on projectives (for $b$ sufficiently
 close to 0), 
\item[(B)] and an embedding $\ind_{b,\lambda}\hookrightarrow
\Ind_{b,0}$.\end{itemize}

\begin{Lem}\label{Lem:ind_embedding1}
There is a natural transformation  $\ind_{0,\lambda}\rightarrow \ind_{b,\lambda}$ as in (A).
\end{Lem}
\begin{proof}
The proof closely follows that of Proposition \ref{Prop_emb2}. We need to show that
for all $b$ sufficiently close to 0 there is a functorial homomorphism 
$E_0(M^{\wedge_0})
\rightarrow E_0(M^{\wedge_b})$ for any $M\in \Ocat^\lambda$, and that
this homomorphism is an embedding whenever $M$ is projective.
This is done exactly as in the proof of Proposition \ref{Prop_emb2}, the only two claims  
that we need to check are that $\Ocat^\lambda$ is a length category with enough projectives,
and that any projective is a free $\K[\h]$-module.
For a $\underline{W}$-module $\mu$
one can define the standard object 
$\Delta^\lambda(\mu)=H\otimes_{S\h\#\underline{W}}\mu\cong \K[\h]\otimes (\K W\otimes_{\K\underline{W}}\mu)$,
where $S\h$ acts on $\mu$ via $\lambda$.
The functor $(\widetilde{\vartheta}_\lambda)_*^{-1}\circ I\circ \zeta_\lambda^{-1}\cong (\widetilde{\theta}_\lambda)_*^{-1}\circ I\circ \zeta_\lambda^{-1}$ defines an
equivalence $\underline{\Ocat}^+\rightarrow \Ocat^\lambda$. It is easy to check
that this equivalence maps standards to standards. Since any projective in $\underline{\Ocat}^+$
admits a filtration, whose quotients are standards, we see that any projective in $\Ocat^\lambda$
is free as a $\K[\h]$-module.
\end{proof}

To establish an embedding $\ind_{b,\lambda}\hookrightarrow \Ind_{b,0}$ we argue as in Subsections \ref{SUBSECTION_Functors_R},\ref{SUBSECTION_embedding2}. Namely, we introduce a category
$\widetilde{\Ocat}_R$ of certain $H_R$-modules equipped with an operator, compare with the
definition of $\underline{\Ocat}_R^+$ in Subsection \ref{SUBSECTION_Functors_R}. Then we
construct functors $\Ind_{b,\lambda/t},\ind_{b,\lambda/t}: \underline{\Ocat}^+\rightarrow \widetilde{\Ocat}_R$ as follows:
\begin{align}\label{eq:def:Ind_t}
&\Ind_{b,\lambda/t}(N)= E_0\circ(\theta_b)^{-1}_*\circ I\circ(\zeta_{\lambda/t}\circ\underline{E}_{\lambda/t})^{-1}(R\otimes N),
\\&\ind_{b,\lambda/t}(N)=E_0 (\left((\widetilde{\theta}_{\lambda/t})^{-1}_*\circ I\circ \zeta_{\lambda/t}(R\otimes N)\right)^{\wedge_b}).
\end{align}
To get an operator $\eu_t^N$ on, say, $\Ind_{b,\lambda/t}$, we reverse the procedure of obtaining $\underline{\eu}_t^{+M}$
from $\eu_M^t$, see Subsection \ref{SUBSECTION_Functors_R}. We also remark that in (\ref{eq:def:Ind_t})
we view $\zeta_{\lambda/t}\circ\underline{E}_{\lambda/t}$ as an equivalence  $\underline{\Ocat}_R'\xrightarrow{\sim}
\underline{\Ocat}^+_R$, see Subsection \ref{SUBSECTION_Functors_R}.

By Lemma \ref{Lem:fun_iso21}, $\Ind_{b,0/t}\cong \Ind_{b,\lambda/t}$. Next, similarly to the corresponding
argument in Subsection \ref{SUBSECTION_Functors_R}, we can construct a map
$$\ind_{b,\lambda/t}(N)\rightarrow 
(\theta_b)^{-1}_*\circ I\circ(\zeta_{\lambda/t}\circ\underline{E}_{\lambda/t})^{-1}(R\otimes N).$$ 
This map is obtained
by applying $\exp(X_{12})$. As in Subsection \ref{SUBSECTION_Functors_R}, this map
gives rise to an isomorphism $\ind_{b,\lambda/t}\xrightarrow{\sim} \Ind_{b,\lambda/t}$.

The relation between the functors $\Ind_{b,0}$ and $\Ind_{b,0/t}$ is completely analogous
to that between $\Res_{b,0}$ and $\Res_{b,0/t}$ (see Subsection \ref{SUBSECTION_embedding2}).
Namely, $\Res_{b,0/t}(N)=R\otimes \Res_{b,0}(N)$, and $\Res_{b,0}(N)$ is the quotient
of the submodule of $\eu+t\frac{d}{dt}$-finite elements in $\Res_{b,0/t}(N)$ by $t-1$.

Now let us relate $\ind_{b,\lambda}$ to $\ind_{b,\lambda/t}$.

Set $M:=(\widetilde{\theta}_\lambda)_*^{-1}\circ I\circ \zeta_{\lambda}^{-1}(N)$.
First of all, let us identify $M_t:=(\widetilde{\theta}_\lambda)_*^{-1}\circ I\circ \zeta_{\lambda/t}^{-1}(R\otimes N)$
with $R\otimes  M$. Namely, it is easy to see that $M_t$ gets identified with $R\otimes M$
if we modify the $H_R$-module structure on $R\otimes M$ as follows: $t\cdot m=tm, w\cdot m=wm,
x\cdot m=tx m, y\cdot m=t^{-1}ym$, $m\in M, w\in W, x\in \h^*,y\in \h$, where in the l.h.s.
we have the new action of $H_R$, and in the r.h.s. the action is standard.
Under this identification $M_t^{\wedge_b}$ gets identified with $(R\otimes M)^{\wedge_{bt}}:=
R[\h]^{\wedge_{bt}}\otimes_{\K[\h]}M$. So we need to relate the $H$-module $E_0(M^{\wedge_b})$
to the $H_R$-module $E_0((R\otimes M)^{\wedge_{bt}})$.

The module $(R\otimes M)^{\wedge_{bt}}$ comes equipped with an Euler operator
$\eu_t^M$ induced from $t\frac{d}{dt}$ on $R\otimes M$. We remark that the maximal
ideal $\m_{bt}\subset R[\h]^W$ is stable with respect to $[\eu,\cdot]+t\frac{d}{dt}$.
Consider the quotient $M_n:=(R\otimes M)/\m_{bt}^n$. We want to  understand the structure
of the operator $\eu+\eu_t^M$ on $M_n$. Recall the section $\varphi:\K/\Z\rightarrow \K$
chosen in Subsection \ref{SUBSECTION_embedding2}.

\begin{Lem}\label{Lem:eigenv_3}
There are finitely many elements $\beta_1,\ldots,\beta_k\in \varphi(\K/\Z)$ with $\beta_i-\beta_j\not\in \Z$
such that the $\eu+\eu_t^M$-finite part of
$M_n$ is the direct sum of generalized eigenspaces of $\eu+\eu_t^M$ with eigenvalues $\beta_i+n, n\in \Z$.
All generalized eigenspaces are finite dimensional. Moreover, if $M_n^0$ denotes the sum of
generalized eigenspaces with eigenvalues $\beta_1,\ldots,\beta_k$, then the natural homomorphism
$R\otimes M_n^0\rightarrow M_n$ is an isomorphism.
\end{Lem}
\begin{proof}
Consider the $\K[[t]]\otimes H$-module $\K[[t]]\otimes M$ that maps naturally to $R\otimes M$. Then we have
$R\otimes_{\K[[t]]}M_n^+\xrightarrow{\sim }M_n$, where $M_n^+:=\K[[t]]\otimes M/\m_{bt}^n$.  The operator 
$\eu+\eu_M^t$ also acts naturally on $M_n^+,R\otimes_{\K[[t]]}M_n^+$ and the identification 
$R\otimes_{\K[[t]]}M_n^+\xrightarrow{\sim }M_n$ intertwines the corresponding operators.
The $\K[[t]]$-module $M_n^+$ is finitely generated because $\K[[t]]\otimes M$
is finitely generated over $\K[[t]]\otimes\K[\h]$. Hence $M_n^+$ is complete in the $t$-adic topology. All subspaces
$t^m M_n^+$ are $\eu+\eu_t^M$-stable. The claim of the lemma follows easily from the observation
that all quotients $M_n^+/t^m M_n^+$ are finite dimensional over $\K$ and that the multiplication by
$t$ increases the eigenvalue of $\eu+\eu_t^M$ by 1.
\end{proof}

The proof also shows that the subspace $(M_n^+)_{fin}$
of $\eu+\eu^t_M$-finite vectors in $M_n^+$ coincides with  $(\K[t]\otimes M)/\m_{bt}^n$
embedded naturally into $M_n^+$ (the natural map is an embedding because
$M_n=\K[[t]]\otimes_{\K[t]}(\K[t]\otimes M)/\m_{bt}^n$ and the torsion submodule
of the $\K[t]$-module $(\K[t]\otimes M)/\m_{bt}^n$ is supported at 0 thanks to
the operator $\eu+\eu^M_t$).

Let us identify $M_n^0$ with $M/\m_b^n$.
For $m\gg 0$ the space $t^m M_n^0$ lies in the torsion-free
part of the $\K[[t]]$-module $M_n^+$. Then we just consider $t^m M_n^0$ as a subspace
in $(\K[t]\otimes M)/\m_{bt}^n$ and restrict the natural projection (=the quotient by $t-1$)
$(\K[t]\otimes M)/\m_{bt}^n\twoheadrightarrow M/\m_b^n$ to $t^m M_n^0$. From Lemma \ref{Lem:eigenv_3}
it follows easily that this map is a bijection. We remark that the bijection
$M_n^0\rightarrow M/\m_b^n$ is compatible with the natural projections
$M_{n+1}^0\twoheadrightarrow M_{n}^0, M/\m_{b}^{n+1}\twoheadrightarrow M/\m_b^n$
(the claim that the first map is surjective is an easy corollary of Lemma \ref{Lem:eigenv_3}).

It follows that $M^{\wedge_b}$ gets identified with the sum of generalized eigenspaces of
elements of $\varphi(\K/\Z)$ in $M_t^{\wedge_{bt}}$. This is an $H$-module
identification (where   $H$ acts on the latter space  by
$x\mapsto t^{-1}x, y\mapsto ty, w\mapsto w$). From here it is easy
to see that $E_0(M^{\wedge_b})$ gets embedded into $E_0(M^{\wedge_{bt}})$.
This embedding produces an embedding $\ind_{b,\lambda}\hookrightarrow \Ind_{b,0}$
we need.


\begin{thebibliography}{99}
\bibitem[BE]{BE} R. Bezrukavnikov, P. Etingof. {\it Parabolic induction and restriction
functors for rational Cherednik algebras}.  Selecta Math.,  14(2009), 397-425.
\bibitem[BG]{BG} K.A. Brown, I. Gordon. {\it Poisson orders, representation theory and
symplectic reflection algebras}. J. Reine Angew. Math. 559(2003), 193-216.
\bibitem[EG]{EG} P. Etingof, V. Ginzburg. {\it Symplectic reflection algebras, Calogero-Moser space,
and deformed Harish-Chandra homomorphism}. Invent. Math. 147 (2002), N2, 243-348.
\bibitem[GGOR]{GGOR} V. Ginzburg, N. Guay, E. Opdam and R. Rouquier, {\it On the category $\Ocat$ for rational
Cherednik algebras}, Invent. Math., 154 (2003), 617-651.
\bibitem[L]{sraco} I. Losev. {\it Completions of symplectic reflection
algebras}. arXiv:1001.0239v2.
\bibitem[S]{Shan} P. Shan. {\it Crystals of Fock spaces and cyclotomic rational double affine Hecke algebras}.
arXiv:0811.4549.
\end{thebibliography}
\end{document}